\documentclass[12pt,oneside,reqno]{amsart}
\usepackage{amssymb}
\usepackage[active]{srcltx}
\usepackage{a4wide}
\usepackage{amsmath,amsthm}
\usepackage{amssymb}
\usepackage{amstext}
\usepackage{cite}
\usepackage{epsfig}
\usepackage{enumerate}
\usepackage{color}
\usepackage{amssymb}
\newtheorem{theorem}{Theorem}[section]
\newtheorem{lemma}[theorem]{Lemma}

\theoremstyle{definition}

\theoremstyle{remark}
\newtheorem{remark}[theorem]{Remark}

\usepackage[linkcolor=blue, urlcolor=blue, citecolor=blue, 
colorlinks, bookmarks]{hyperref}

\newcommand{\be}{\begin{equation}}
\newcommand{\ee}{\end{equation}}

\begin{document}

\title{ Fractal Interpolation Function On Products of the Sierpi\'nski Gaskets }



\author{S. A. Prasad}
\address{Department of Mathematics \& Statistics, IIT Tirupati, Tirupati, India 517506}
\email{srijanani@iittp.ac.in}
\author{S. Verma}
\address{Department of Applied Sciences, IIIT Allahabad, Prayagraj, India 211015}
\email{saurabhverma@iiita.ac.in}






\begin{abstract}
In this paper, we aim to construct fractal interpolation function(FIF) on the product of two Sierpi\'nski gaskets. Further, we collect some results regarding smoothness of the constructed FIF. We prove, in particular, that the FIF are H\"older functions under specific conditions. In the final section, we obtain some bounds on the fractal dimension of FIF.
\end{abstract}

\maketitle



\section{Introduction} 
Barnsley \cite{MF2} introduced the concept of Fractal Interpolation Functions (FIFs) on the interval using the theory of Iterated Function System (IFS). Since then, FIFs have proven to be an invaluable tool for interpolating experimental data by a non-smooth curve, with numerous applications in engineering~\cite{baldo94}, biological sciences~\cite{havlina95}, planetary science~\cite{iovane04} and arts~\cite{voss89}. FIFs of various types, such as Hidden-variable FIFs, Coalescence Hidden-variable FIFs, Hermite FIFs, Spline FIFs, and Super FIFs, have been built in~\cite{barnsley89_2, chand07, navascues04, navas-sebas04, srijanani15_1, srijanani21} and properties such as smoothness, approximation property, regularity, multiresolution analysis, reproducing kernel, node insertion, fractional calculus, dimension property have been studied in~\cite{SS2, gang96, navascues03, srijanani14_1, srijanani13_2, srijanani14_2, srijanani13_1,pan14, srijanani17, ruan09, AKBCK, Ver1}. The development of Fractal Interpolation Surfaces or multivariable FIFs formed by employing higher dimensional or recurrent IFSs has been addressed in~\cite{massopust90,geronimo93,massopust93,dalla02}.

Celik et al. \cite{Celik} defined FIFs on the Sierpinski Gasket (SG), a well-known fractal domain. Further, these FIFs were studied  by Ruan \cite{Ruan3} on the basis of p.c.f. self-similar sets. In \cite{Ruan3}, Ruan extended his work on p.c.f. self-similar sets which were introduced by Kigami by defining the FIFs on these sets. On $SG$, Ri and Ruan \cite{Ruan4} explored some properties of uniform fractal interpolation functions, a particular class of FIFs. First, they investigated the min-max property of uniform FIFs. Then, they determined a necessary and sufficient condition for uniform FIFs to have finite energy. Uniform FIFs’s normal derivatives and laplacian were also discussed.  Sahu and Priyadarshi \cite{SP} computed  some bounds for box dimensions of harmonic functions on the $SG$. They also constructed a fractal interpolation function on the $SG$. Recently, Agrawal and Som \cite{EPJST} estimated box dimension of FIF on the $SG$. Further, they \cite{RIM} used FIFs to approximate functions on $SG.$ In \cite{NVV2}, Navascu\'es et al. constructed vector-valued FIFs on SG and studied several approximation theoretic results. Strichartz \cite{BS,Stri} started analysis on the products of SGs and developed a new concept of partial differential equations on the products of SGs. We believe that analogues to Ruan's work \cite{Ruan3,Ruan4} on FIFs on SG and Navascu\'es's work \cite{navascues03,navascues04} on fractal approximation, our work will find several future applications.

The purpose of this research is to develop fractal interpolation functions on the product of two Sierpi\'nski gaskets. The smoothness and bounds on box-counting dimension are obtained after creating the FIF on the product of two Sierpi\'nski gaskets. The organization of the paper is as follows: Section~\ref{sec:constr}  provides a brief introduction to the Iterated Function System and Attractor. The same Section also describes the construction of a Fractal Interpolation Function on the product of two Sierpi\'nski gaskets. Section~\ref{sec:sm} contains an instructive description of the smoothness of a FIF formed on the product of two Sierpi\'nski gaskets. Finally, Section~\ref{sec:fd} discusses the bounds of Fractal dimension of the FIF built on the product of two Sierpi\'nski gaskets. 

\section{Fractal interpolation function on $SG' \times SG''$}\label{sec:constr}
Let $\{X;W_i,i=1,2,\dots,k\}$ be an Iterated Function System (IFS), where $(X,d) $ is a complete metric space and $W_i:X \to X$ are contractive mappings with contraction ratio $\alpha_i$ respectively.
The IFS generates the mapping $W$ from $\mathcal{K}(X)$ into $\mathcal{K}(X)$ given by
$$ \mathcal{W}(K) = \cup_{i=1}^k W_i (K),$$
where $\mathcal{K}(X)$ is the collection of all non-empty compact subsets of $X.$
The Hutchinson-Barnsley map $\mathcal{W}$ defined above is then a contraction
mapping, with respect to the Hausdorff metric $h_d$, the contraction ratio $\alpha$ of $\mathcal{W}$ is equal to $\max\{\alpha_1, \alpha_2, \dots, \alpha_k\}$. Then, by the Banach
contraction principle, there exists a unique nonempty compact subset $K_*$ such that
$K_*= \cup_{i=1}^k W_i(K_*)$, The set $K_*$ is termed as the attractor of the IFS. For further details, the reader is referred to \cite{MF2}.
\par
Let $U_0=\{p_1,p_2,p_3\}$ be the vertices of an equilateral triangle on the plane $\mathbb{R}^2$ and $L_i(t)=\frac{1}{2}( t + p_i),\ i=1,2,3$ be three contractions of the plane which constitute an iterated functions system. For $n \in \mathbb{N},$ we denote the collection of all words with length $n$ by $\{1,2,3\}^n,$ i.e. if $\boldsymbol{\omega} \in \{1,2,3\}^n$ then $\boldsymbol{\omega}=\omega_1,\omega_2, \dots ,\omega_n$ where $\omega_i \in \{1,2,3\}.$ We define, for $\boldsymbol{\omega} \in \{1,2,3\}^N$, $$L_{\boldsymbol{\omega}}= L_{\omega_1} \circ L_{\omega_2} \circ \dots \circ L_{\omega_N}.$$ $SG'$ is the attractor of the system:
     $$ SG'= L_1(SG') \cup L_2(SG') \cup L_3(SG').$$
Similarly, consider $W_0=\{q_1,q_2,q_3\}$ be the vertices of another equilateral triangle on another $\mathbb{R}^2$ plane and $K_i(t)=\frac{1}{2}( t + q_i), i=1,2,3$ be three contractions of the second plane constituting another iterated functions system. Then, we have another attractor $$ SG''= K_1(SG'') \cup K_2(SG'') \cup K_3(SG'').$$ 
Hence, we have $$ SG' \times SG''= \cup_{i=1,j=1}^3 L_i(SG') \times K_j(SG'').$$
Iteratively, we get $$ SG' \times SG''= \cup_{\boldsymbol{\omega} ,\boldsymbol{\eta}\in I^n} L_{\boldsymbol{\omega}}(SG') \times K_{\boldsymbol{\eta}}(SG'').$$
       
Define $V_N$ by $V_N=\{(L_{\boldsymbol{\omega}}(p_i),K_{\boldsymbol{\eta}}(q_j)):\boldsymbol{\omega},\boldsymbol{\eta} \in I^N,~ i, j \in I\}.$ The union of images of $V_{0}:=U_0 \times W_0$ under these iterations $(L_{\boldsymbol{\omega}}(.), K_{\boldsymbol{\eta}}(.)) $ with $|w|=N$ constitutes the set of N-th stage vertices $V_N$ of $SG' \times SG''$. Further, define $V_{*}= \cup_{N=1} ^{\infty}V_N.$

Let $\{(y,z_y):y \in V_N\} \subset SG' \times SG''\times \mathbb{R}$ be a data set such that $z_y$ is zero on the boundary of $SG' \times SG''$. We construct an IFS whose attractor is the graph of a function $f:SG' \times SG'' \to \mathbb{R}$ such that $$f(y)=z_y~~~\text{for all}~ y \in V_N.$$

Let $X=SG' \times SG''\times \mathbb{R}$ and define maps $W_{\boldsymbol{\omega \eta}} :X \rightarrow X$ by $$W_{\boldsymbol{\omega \eta}}(t,s,x)=\Big(L_{\boldsymbol{\omega}}(t),K_{\boldsymbol{\eta}}(s),F_{\boldsymbol{\omega \eta}}(t,s,x)\Big),~~ \boldsymbol{\omega},\boldsymbol{\eta} \in I^N,$$
where $F_{\boldsymbol{\omega \eta}}: X \to \mathbb{R}$ is given by $$ F_{\boldsymbol{\omega \eta}}(t,s,x)= \alpha_{\boldsymbol{\omega \eta}}(t,s) x + h_{\boldsymbol{\omega}\boldsymbol{\eta}}(t,s), ~ \boldsymbol{\omega},\boldsymbol{\eta} \in I^N,$$ where $h_{\boldsymbol{\omega}\boldsymbol{\eta}}: SG' \times SG'' \to \mathbb{R}$ is a function such that
\begin{itemize}
    \item For $\boldsymbol{\omega}= \omega_1\omega_2\dots \omega_{N-1}i,$ $\boldsymbol{\tau}= \omega_1\omega_2\dots \omega_{N-1}j$,  $h_{\boldsymbol{\omega}\boldsymbol{\eta}}(p_j,s)=h_{\boldsymbol{\tau}\boldsymbol{\eta}}(p_i,s)$ for all $s \in SG''.$
    \item For $\boldsymbol{\eta}= \eta_1\eta_2\dots \eta_{N-1}i,$ $\boldsymbol{\xi}= \eta_1\eta_2\dots \eta_{N-1}j$,   $h_{\boldsymbol{\omega}\boldsymbol{\eta}}(t,q_j)=h_{\boldsymbol{\omega}\boldsymbol{\xi}}(t,q_i)$ for all $t \in SG'.$
    \item For $\boldsymbol{\omega}= \omega_1\omega_2\dots \omega_{N-1}i,$ $\boldsymbol{\tau}= \omega_1\omega_2\dots \omega_{N-1}j$, and $\boldsymbol{\eta}= \eta_1\eta_2\dots \eta_{N-1}k,$ $\boldsymbol{\xi}= \eta_1\eta_2\dots \eta_{N-1}l$,  $h_{\boldsymbol{\omega}\boldsymbol{\eta}}(p_j,q_l)=h_{\boldsymbol{\tau}\boldsymbol{\xi}}(p_i,q_k),$ for all $i,j,k,l \in I.$
\end{itemize}
and for each $ \boldsymbol{\omega},\boldsymbol{\eta} \in I^N$, $\alpha_{\boldsymbol{\omega}\boldsymbol{\eta}}: SG' \times SG''\to \mathbb{R}$ is a continuous function with $\|\alpha_{\boldsymbol{\omega}\boldsymbol{\eta}}\|_{\infty} < 1.$ 
     Then $\{X;W_{\boldsymbol{\omega \eta}}, \boldsymbol{\omega},\boldsymbol{\eta}  \in I^N\}$ is an IFS and admits an invariant set. 
   \begin{theorem}\label{MTconstr1}
 The IFS $\{X; W_{\boldsymbol{\omega}\boldsymbol{\eta}}, \boldsymbol{\omega}, \boldsymbol{\eta}  \in I^N\}$ defined above has a unique attractor $Gr(f).$ The set $Gr(f)$ is the graph of a continuous function $f: SG' \times SG''\rightarrow \mathbb{R}$ such that $f(y)=z_y~~~\text{for all}~ y \in V_N.$ The function $f$ is known as the Fractal Interpolation Function on the product $SG' \times SG''.$
\end{theorem}
\begin{proof}
We know that $\mathcal{C}(SG' \times SG'')$ the space of all continuous functions forms a Banach space  with respect to supnorm. Let \begin{equation*}
\begin{aligned}
\mathcal{C}_*(SG' \times SG'')=\Big\{ g \in \mathcal{C}(SG' \times SG''):&g(y)=z_y,~y \in V_N, g(p_i,s)=g(t,q_i)=0,\\& \forall\ t \in SG',~s \in SG'',~i \in I\Big\}.
\end{aligned}
\end{equation*} Define a Read-Bajraterevic  map $T: \mathcal{C}_*(SG' \times SG'') \rightarrow \mathcal{C}_*(SG' \times SG'') $ by $$ (Tg)(t,s)= F_{\boldsymbol{\omega}\boldsymbol{\eta}}\big(L^{-1}_{\boldsymbol{\omega}}(t),K^{-1}_{\boldsymbol{\eta}}(s),g\big(L^{-1}_{\boldsymbol{\omega}}(t),K^{-1}_{\boldsymbol{\eta}}(s)\big)\big), ~~~\forall~ (t,s) \in L_{\boldsymbol{\omega}}(SG') \times K_{\boldsymbol{\eta}}(SG''),$$ 
 for every $~\boldsymbol{\omega},\boldsymbol{\eta} \in I^N.$

We first show that $T$ is well-defined. For this, we need to check the continuity of $Tg$ at the following point :
\begin{itemize}
    \item $(t,s) \in (L_{\boldsymbol{\omega}} (SG') \cap L_{\boldsymbol{\tau}} (SG')) \times K_{\boldsymbol{\eta}}(SG'')$
    \item $(t,s) \in L_{\boldsymbol{\omega}}(SG') \times (L_{\boldsymbol{\eta}} (SG'') \cap L_{\boldsymbol{\xi}} (SG'')) $
    \item $(t,s) \in (L_{\boldsymbol{\omega}} (SG') \cap L_{\boldsymbol{\tau}} (SG'))\times (L_{\boldsymbol{\eta}} (SG'') \cap L_{\boldsymbol{\xi}} (SG'')).$
\end{itemize}

From the construction of $SG'$ and $SG'',$ for any $t \in L_{\boldsymbol{\omega}} (SG') \cap L_{\boldsymbol{\tau}} (SG'),$ one may get the following relation: $\boldsymbol{\omega}= \omega_1\omega_2\dots \omega_{N-1}i,$ $\boldsymbol{\tau}= \omega_1\omega_2\dots \omega_{N-1}j$ and $t = L_{\boldsymbol{\omega}}(p_j)=L_{\boldsymbol{\tau}} (p_i)$ for some $i,j \in I.$
Now, since $g(p_k,s)=0, \forall s \in SG'',~k \in I,$ we have 
\begin{equation*}
\begin{aligned}
(Tg)(t,s)& = F_{\boldsymbol{\omega}\boldsymbol{\eta}}\big(L^{-1}_{\boldsymbol{\omega}}(t),K^{-1}_{\boldsymbol{\eta}}(s),g(L^{-1}_{\boldsymbol{\omega}}(t),K^{-1}_{\boldsymbol{\eta}}(s))\big)\\ & = F_{\boldsymbol{\omega}\boldsymbol{\eta}}\big( p_j,K^{-1}_{\boldsymbol{\eta}}(s), g(p_j,K^{-1}_{\boldsymbol{\eta}}(s))\big)\\& =\alpha_{\boldsymbol{\omega \eta}}(p_j,K^{-1}_{\boldsymbol{\eta}}(s)) g(p_j,K^{-1}_{\boldsymbol{\eta}}(s)) + h_{\boldsymbol{\omega}\boldsymbol{\eta}}(p_j,K^{-1}_{\boldsymbol{\eta}}(s))\\&=h_{\boldsymbol{\omega}\boldsymbol{\eta}}(p_j,K^{-1}_{\boldsymbol{\eta}}(s)),
\end{aligned}
\end{equation*}
and 
\begin{equation*}
\begin{aligned}
(Tg)(t,s)& = F_{\boldsymbol{\tau}\boldsymbol{\eta}}\big(L^{-1}_{\boldsymbol{\tau}}(t),K^{-1}_{\boldsymbol{\eta}}(s),g(L^{-1}_{\boldsymbol{\tau}}(x),K^{-1}_{\boldsymbol{\eta}}(s))\big)\\ & = F_{\boldsymbol{\tau}\boldsymbol{\eta}}\big( p_i,K^{-1}_{\boldsymbol{\eta}}(s), g(p_i,K^{-1}_{\boldsymbol{\eta}}(s))\big)\\& =\alpha_{\boldsymbol{\tau \eta}}(p_i,K^{-1}_{\boldsymbol{\eta}}(s)) g(p_i,K^{-1}_{\boldsymbol{\eta}}(s)) + h_{\boldsymbol{\tau}\boldsymbol{\eta}}(p_i,K^{-1}_{\boldsymbol{\eta}}(s))\\&=h_{\boldsymbol{\tau}\boldsymbol{\eta}}(p_i,K^{-1}_{\boldsymbol{\eta}}(s)).
\end{aligned}
\end{equation*}
With the help of conditions on $h_{\boldsymbol{\omega}\boldsymbol{\eta}}$, it follows that $Tg$ is continuous at $(t,s) \in (L_{\boldsymbol{\omega}} (SG') \cap L_{\boldsymbol{\tau}} (SG')) \times K_{\boldsymbol{\eta}}(SG'')$. Other cases can be dealt similarily. Further, since $F_{\boldsymbol{\omega}\boldsymbol{\eta}}$ is continuous for each $\boldsymbol{\omega},\boldsymbol{\eta} \in I^N$, $Tg$ is continuous, and $Tg \in \mathcal{C}_*(SG' \times SG'').$ Hence, $T$ is well-defined. 
Now, let $g,h \in \mathcal{C}_*(SG' \times SG'').$ Then we have 
\begin{equation*}
\begin{aligned}
|(Tg)(t,s)-(Th)(t,s)| &= \Big|F_{\boldsymbol{\omega}\boldsymbol{\eta}}\big(L^{-1}_{\boldsymbol{\omega}}(t),K^{-1}_{\boldsymbol{\eta}}(s),g\big(L^{-1}_{\boldsymbol{\omega}}(t),K^{-1}_{\boldsymbol{\eta}}(s)\big)\big)\\&~~~~~ - F_{\boldsymbol{\omega}\boldsymbol{\eta}}\big(L^{-1}_{\boldsymbol{\omega}}(t),K^{-1}_{\boldsymbol{\eta}}(s),h\big(L^{-1}_{\boldsymbol{\omega}}(t),K^{-1}_{\boldsymbol{\eta}}(s)\big)\big)\Big|\\& \le  |\alpha_{\boldsymbol{\omega \eta}}(t,s)| \Big|  g\big(L^{-1}_{\boldsymbol{\omega}}(t),K^{-1}_{\boldsymbol{\eta}}(s)\big) -h\big(L^{-1}_{\boldsymbol{\omega}}(t),K^{-1}_{\boldsymbol{\eta}}(s)\big)\Big|\\& \le  \|\alpha_{\boldsymbol{\omega \eta}}\|_{\infty} \|  g -h\|_{\infty}.
\end{aligned}
\end{equation*}
This in turn yields 
\[
\|Tg-Th\|_{\infty} \le \|\alpha_{\boldsymbol{\omega \eta}}\|_{\infty} \|  g -h\|_{\infty}.
\]
Since $\|\alpha_{\boldsymbol{\omega \eta}}\|_{\infty}<1$, $T$ is a contraction. Using the Banach fixed point theorem, $T$ has a unique fixed point $f \in \mathcal{C}_*(SG' \times SG'').$ Since $Tf=f,$
\begin{equation}\label{functional}
\begin{aligned} f(t,s)= F_{\boldsymbol{\omega}\boldsymbol{\eta}}\big(L^{-1}_{\boldsymbol{\omega}}(t),K^{-1}_{\boldsymbol{\eta}}(s),f\big(L^{-1}_{\boldsymbol{\omega}}(t),K^{-1}_{\boldsymbol{\eta}}(s)\big)\big), ~~~~~\forall~ (t,s) \in L_{\boldsymbol{\omega}}(SG') \times K_{\boldsymbol{\eta}}(SG''), 
\end{aligned}
\end{equation}
for every $~\boldsymbol{\omega},\boldsymbol{\eta} \in I^N.$
By using the proof of Theorem $3.8$ in \cite{Jha} (motivated by Barnsley \cite{MF2}) we can show that attractor associated with mentioned IFS is the graph of the above function. Now, we show that it is the graph of FIF $f$. With the help of functional equation (\ref{functional}) and $ SG' \times SG''= \cup_{\boldsymbol{\omega} ,\boldsymbol{\eta}\in I^n} L_{\boldsymbol{\omega}}(SG') \times K_{\boldsymbol{\eta}}(SG'')$, we have
\begin{equation*}
    \begin{aligned}
   \cup_{\boldsymbol{\omega},\boldsymbol{\eta} \in I^N} W_{\boldsymbol{\omega}\boldsymbol{\eta}}(Gr(f))= &   \cup_{\boldsymbol{\omega},\boldsymbol{\eta} \in I^N} \Big\{\big(L_{\boldsymbol{\omega}}(t),K_{\boldsymbol{\eta}}(s),f(L_{\boldsymbol{\omega}}(x),K_{\boldsymbol{\eta}}(s))\big):x \in I \Big\}\\
    = & \cup_{\boldsymbol{\omega},\boldsymbol{\eta} \in I^N} \big\{(t,s,f(t,s)):(t,s) \in  L_{\boldsymbol{\omega}}(SG') \times K_{\boldsymbol{\eta}}(SG'') \big\}\\
    =& Gr(f),
    \end{aligned}
\end{equation*} completing the proof.
\end{proof}
\begin{remark}
The assumption $g(p_i,s)=g(t,q_i)=0, \forall\ t \in SG',~s \in SG'',~i \in I$ can be weakened by taking constant and same scaling functions, that is, if  $\alpha_{\boldsymbol{\omega \eta}}= \alpha$ then $g(p_i,s)=g(p_j,s)$ and $g(t,q_i)=g(t,q_j)$ for all $\forall\ t \in SG',~s \in SG'',~i,j \in I$ will be sufficient to show that the mapping $T$ is well-defined. It may hint us that we can have many different constructions of FIF on $SG'\times  SG''$ as we have for fractal surfaces \cite{massopust90,massopust93,Ruan}.
\end{remark}

\section{Smoothness of Fractal Interpolation Functions on $SG \times SG$}\label{sec:sm}
In this section, we aim to discuss the smoothness property of the FIF $f: SG' \times SG'' \to \mathbb{R}$ obtained through the IFS:
$$L_{\boldsymbol{\omega}}(t)=\frac{1}{2^N}t + \sum_{k=1}^{N} \frac{1}{2^{k}}p_{\omega_k};  K_{\boldsymbol{\eta}}(s)=\frac{1}{2^N}s + \sum_{k=1}^{N} \frac{1}{2^{k}}q_{\eta_k}; ~~    F_{\boldsymbol{\omega \eta}}(t,s,x)= \alpha_{\boldsymbol{\omega \eta}}(t,s) x + h_{\omega \eta}(t,s), ~ \boldsymbol{\omega} \in I^N.$$ Targeting to give a simplified presentation, one may introduce the notation $a=\frac{1}{2^N}$, $a_{\boldsymbol{\omega}}=\sum_{k=1}^{N} \frac{1}{2^{k}}p_{\omega_k}$ and $b_{\boldsymbol{\omega}}=\sum_{k=1}^{N} \frac{1}{2^{k}}q_{\omega_k}$. Then we have $L_{\boldsymbol{\omega}}(t)=at+a_{\boldsymbol{\omega}}, ~K_{\boldsymbol{\omega}}(s)=as+b_{\boldsymbol{\omega}}$ for $ \boldsymbol{\omega} \in I^N.$ In this section, for computational convenience, we consider $\|(t,s) -(t',s')\|_2 \le 1 $ for all $(t,s), (t',s') \in SG' \times SG''.$
\par 

We show that for suitable choices of the parameters, the fractal function $f$ preserves the H\"{o}lder continuity of FIF $h$. 

For $t \in SG'$ and $\boldsymbol{\omega}_j \in I^N$, let $$L_{\boldsymbol{\omega}_1 \boldsymbol{\omega}_2\dots \boldsymbol{\omega}_m}(t)=L_{\boldsymbol{\omega}_1}\circ L_{\boldsymbol{\omega}_2} \circ \dots \circ L_{\boldsymbol{\omega}_m}(t), ~~L_{\boldsymbol{\omega}_1 \boldsymbol{\omega}_2\dots \boldsymbol{\omega}_m}(SG')=L_{\boldsymbol{\omega}_1}\circ L_{\boldsymbol{\omega}_2} \circ \dots \circ L_{\boldsymbol{\omega}_m}(SG' ) .$$ Define a shift operator $\sigma$ by $$ \sigma(\boldsymbol{\omega}_1 \boldsymbol{\omega}_2 \dots \boldsymbol{\omega}_m)=(\boldsymbol{\omega}_2 \boldsymbol{\omega}_3 \dots \boldsymbol{\omega}_m).$$ Let $ \sigma^k $ denote the $k-$fold autocomposition of $\sigma$ such that for $1 \le k \le m-1$, $L_{\sigma ^k(\boldsymbol{\omega}_1 \boldsymbol{\omega}_2\dots \boldsymbol{\omega}_m)}(t)= L_{\boldsymbol{\omega}_{k+1} \boldsymbol{\omega}_{k+2} \dots \boldsymbol{\omega}_m}(t)$  otherwise $L_{\sigma ^k(\boldsymbol{\omega}_1 \boldsymbol{\omega}_2\dots \boldsymbol{\omega}_m)}(t)=t.$ With the help of the successive iteration and induction, we may prove the following lemma; see also \cite{WY}.
\begin{lemma} \label{Smoothlem}
Let $f$ be a fractal function corresponding to the aforementioned IFS. For any $(t,s) \in SG' \times SG''$ and $\boldsymbol{\omega}_j \in I^N,$
\begin{equation*}
           \begin{aligned}
             L_{\boldsymbol{\omega}_1 \boldsymbol{\omega}_2\dots \boldsymbol{\omega}_m}(t)=a^{m} t + \sum_{k=1}^{m} a^{k-1} a_{\boldsymbol{\omega}_k}~\text{and}~ K_{\boldsymbol{\eta}_1 \boldsymbol{\eta}_2\dots \boldsymbol{\eta}_m}(s)=a^{m} s + \sum_{k=1}^{m} a^{k-1} b_{\boldsymbol{\eta}_k}
              \end{aligned}
           \end{equation*}
           and
  \begin{equation*}
             \begin{aligned}
               &f \big(L_{\boldsymbol{\omega}_1 \boldsymbol{\omega}_2\dots \boldsymbol{\omega}_m}(t), K_{\boldsymbol{\eta}_1 \boldsymbol{\eta}_2\dots \boldsymbol{\eta}_m}(s)\big)\\=& \Big(\prod_{k=1}^{m} \alpha_{\boldsymbol{\omega}_k\boldsymbol{\eta}_k}(L_{\sigma^k(\boldsymbol{\omega}_1 \boldsymbol{\omega}_2\dots \boldsymbol{\omega}_m) }(t),K_{\sigma^k(\boldsymbol{\eta}_1 \boldsymbol{\eta}_2\dots \boldsymbol{\eta}_m) }(s))\Big) f(t,s)\\&+ \sum_{r=1}^{m} \Big(\prod_{k=1}^{r-1} \alpha_{\boldsymbol{\omega}_k \boldsymbol{\eta}_k}(L_{\sigma^k(\boldsymbol{\omega}_1 \boldsymbol{\omega}_2\dots \boldsymbol{\omega}_m) }(t),K_{\sigma^k(\boldsymbol{\eta}_1 \boldsymbol{\eta}_2\dots \boldsymbol{\eta}_m) }(s))\Big)h_{\boldsymbol{\omega}_r\boldsymbol{\eta}_r}\big(L_{\sigma ^{r-1}(\boldsymbol{\omega}_1 \boldsymbol{\omega}_2\dots \boldsymbol{\omega}_m)}(t),K_{\sigma ^{r-1}(\boldsymbol{\eta}_1 \boldsymbol{\eta}_2\dots \boldsymbol{\eta}_m)}(s)\big)
               ,
                \end{aligned}
             \end{equation*}
  where
  \begin{equation*}
               \begin{aligned}
                 L_{\sigma ^k(\boldsymbol{\omega}_1 \boldsymbol{\omega}_2\dots \boldsymbol{\omega}_m)}(t)= a^{m-k} t + \sum_{l=1}^{m-k} a^{l-1} a_{\boldsymbol{\omega}_{k+l}}~\text{and}~K_{\sigma ^k(\boldsymbol{\eta}_1 \boldsymbol{\eta}_2\dots \boldsymbol{\eta}_m)}(s)= a^{m-k} s + \sum_{l=1}^{m-k} a^{l-1} b_{\boldsymbol{\omega}_{k+l}}.
                  \end{aligned}
               \end{equation*}
\end{lemma}

Let us now state the next result borrowed from \cite{WY}.
\begin{lemma}\label{Inequality}
Let $r_i, q_i,~~ i=1,2,\dots, m,$ be  given real numbers. Then  $$ \prod_{i=1}^{m} r_i- \prod_{i=1}^{m} q_i= \sum_{i=1}^{m} \Big( \prod_{k=1}^{m-1}c_k^{(i)}\Big)( r_i - q_i),$$
where, for all $k=1,2, \dots, m-1,$ each $ c_k^{(i)}, i=1,2, \dots m,$ is a real number from the set $$ \{r_1,r_2, \dots r_{i-1},r_{i+1}, \dots, r_m, q_1, q_2, \dots , q_{i-1},q_{i+1}, \dots, q_m\}.$$
\end{lemma}

\begin{theorem} \label{Holder}
Let the parameter maps $h_{\boldsymbol{\omega}\boldsymbol{\eta}}$, $\alpha_{\boldsymbol{\omega}\boldsymbol{\eta}}$ ($\boldsymbol{\omega},\boldsymbol{\eta} \in I^N$) be H\"{o}lder continuous functions with constants $K_h$, $K_{\alpha}$ and H\"{o}lder exponents $s_h$,  $s_{\alpha_{\boldsymbol{\omega}\boldsymbol{\eta}}}$ respectively. Define $\| \alpha \|_{\infty}= \max\{\|\alpha_{\boldsymbol{\omega}\boldsymbol{\eta}}\|_{\infty}: \boldsymbol{\omega},\boldsymbol{\eta} \in I^N \}$, $\delta = \frac{\| \alpha \|_{\infty}}{a}$ and $s_\alpha= \min \{ s_{\alpha_{\boldsymbol{\omega}\boldsymbol{\eta}}}:  \boldsymbol{\omega},\boldsymbol{\eta} \in I^N \}$. Then the following hold:
\begin{itemize}
\item[(1)] If $\|\alpha\|_{\infty} < \frac{1}{2^N}$, then $f$ is a H\"older continuous function with exponent $s$, where $s=\min\{s_h,s_{\alpha}\}.$
\item[(2)] If $\|\alpha\|_{\infty} = \frac{1}{2^N}$, then $f$ is a H\"older continuous function with exponent $s-\mu $ for some $ 0<\mu<1.$
\item[(3)] If $\|\alpha\|_{\infty} > \frac{1}{2^N}$, then $f$ is a H\"older continuous function with exponent $\lambda,$ where $0< \lambda \le s-1+\frac{\ln\|\alpha\|_{\infty}}{ \ln a} < 1.$
\end{itemize}

\end{theorem}
\begin{proof}
For any $(t,s),(t',s') \in SG' \times SG'',$ it is possible to find $m \ge 0$ such that $t \in L_{\boldsymbol{\omega}_1 \boldsymbol{\omega}_2\dots \boldsymbol{\omega}_m}(SG' ),~ s \in K_{\boldsymbol{\eta}_1 \boldsymbol{\eta}_2\dots \boldsymbol{\eta}_m}(SG'' )$ and $$a^{m+1} \le \|(t,s) -( t',s')\|_2 \le a^{m}  .$$  If $m=0,$ we set $L_{\boldsymbol{\omega}_1 \boldsymbol{\omega}_2\dots \boldsymbol{\omega}_m}(SG' )=SG'$. Let $t, t' \in L_{\boldsymbol{\omega}_1 \boldsymbol{\omega}_2\dots \boldsymbol{\omega}_m}(SG' ). $ Since $ t \in L_{\boldsymbol{\omega}_1 \boldsymbol{\omega}_2\dots \boldsymbol{\omega}_m}(SG'),~s \in K_{\boldsymbol{\eta}_1 \boldsymbol{\eta}_2\dots \boldsymbol{\eta}_m}(SG')$ there exists elements $\bar{t} \in SG'$, $\bar{s} \in SG''$ such that the following conditions hold due to Lemma \ref{Smoothlem}, 
\begin{equation} \label{EQN1}
           \begin{aligned}
              t=L_{\boldsymbol{\omega}_1 \boldsymbol{\omega}_2\dots \boldsymbol{\omega}_m}(\overline{t})= a^{m}  \overline{t} + \sum_{h=1}^{m} a^{h-1} a_{\boldsymbol{\omega}_h}, ~~s=K_{\boldsymbol{\eta}_1 \boldsymbol{\eta}_2\dots \boldsymbol{\eta}_m}(\overline{s})= a^{m}  \overline{s} + \sum_{h=1}^{m} a^{h-1} b_{\boldsymbol{\eta}_h}
              \end{aligned}
           \end{equation}
and
\begin{equation*}
           \begin{aligned}
            f(t,s) =& f(L_{\boldsymbol{\omega}_1 \boldsymbol{\omega}_2\dots \boldsymbol{\omega}_m}(\overline{t}),K_{\boldsymbol{\eta}_1 \boldsymbol{\eta}_2\dots \boldsymbol{\eta}_m}(\overline{s}))  \\=& \Big(\prod_{k=1}^{m} \alpha_{\boldsymbol{\omega}_k\boldsymbol{\eta}_k}(L_{\sigma^k(\boldsymbol{\omega}_1 \boldsymbol{\omega}_2\dots \boldsymbol{\omega}_m) }(t),K_{\sigma^k(\boldsymbol{\eta}_1 \boldsymbol{\eta}_2\dots \boldsymbol{\eta}_m) }(s))\Big) f(\overline{t},\overline{s})\\&+ \sum_{r=1}^{m} \Big(\prod_{k=1}^{r-1} \alpha_{\boldsymbol{\omega}_k \boldsymbol{\eta}_k}(L_{\sigma^k(\boldsymbol{\omega}_1 \boldsymbol{\omega}_2\dots \boldsymbol{\omega}_m) }(\overline{t}),K_{\sigma^k(\boldsymbol{\eta}_1 \boldsymbol{\eta}_2\dots \boldsymbol{\eta}_m) }(\overline{s}))\Big)h_{\boldsymbol{\omega}_r\boldsymbol{\eta}_r}\big(L_{\sigma ^{r-1}(\boldsymbol{\omega}_1 \boldsymbol{\omega}_2\dots \boldsymbol{\omega}_m)}(\overline{t}),K_{\sigma ^{r-1}(\boldsymbol{\eta}_1 \boldsymbol{\eta}_2\dots \boldsymbol{\eta}_m)}(\overline{s})\big).
              \end{aligned}
           \end{equation*}
From $(\ref{EQN1})$, we may write $\overline{t}$ and $\overline{s}$ as follows:
\begin{equation*}
           \begin{aligned}
               \overline{t} = a^{-m}\Big[   t-\sum_{l=1}^{m} a^{l-1} a_{\boldsymbol{\omega}_l} \Big]~\text{and}~\overline{s} = a^{-m}\Big[   s-\sum_{l=1}^{m} a^{l-1} b_{\boldsymbol{\eta}_h} \Big].
              \end{aligned}
           \end{equation*}
In view of Lemma \ref{Smoothlem} and the above equation, we get 
\begin{equation*}
               \begin{aligned}
                 L_{\sigma ^k(\boldsymbol{\omega}_1 \boldsymbol{\omega}_2\dots \boldsymbol{\omega}_m)}(\overline{t})= a^{-k} \Big[   t -\sum_{l=1}^{m} a^{l-1}  a_{\boldsymbol{\omega}_l} \Big]+ \sum_{l=1}^{m-k} a^{l-1} a_{\boldsymbol{\omega}_{k+l}}
                  \end{aligned}
               \end{equation*}
               and 
               \begin{equation*}
               \begin{aligned}
                 K_{\sigma ^k(\boldsymbol{\eta}_1 \boldsymbol{\eta}_2\dots \boldsymbol{\eta}_m)}(\overline{s})= a^{-k} \Big[   s -\sum_{l=1}^{m} a^{l-1}  b_{\boldsymbol{\eta}_l} \Big]+ \sum_{l=1}^{m-k} a^{l-1} b_{\boldsymbol{\eta}_{k+l}}.
                  \end{aligned}
               \end{equation*}
Similarly, since $(t',s') \in L_{\boldsymbol{\omega}_1 \boldsymbol{\omega}_2\dots \boldsymbol{\omega}_m}(SG') \times K_{\boldsymbol{\eta}_1 \boldsymbol{\eta}_2\dots \boldsymbol{\eta}_m}(SG''),$ there exists $(\overline{t'},\overline{s'}) \in SG' \times SG''$ such that $t',s'$ and $f(t',s')$ have expressions similar to the above. Consequently,
\begin{equation*}
               \begin{aligned}
                 &~ \big|f(t,s)- f(t',s')\big| \\ &~ \le   \Bigg\| \Big(\prod_{k=1}^{m} \alpha_{\boldsymbol{\omega}_k \eta_k}\big(L_{\sigma ^{k}(\boldsymbol{\omega}_1 \boldsymbol{\omega}_2\dots \boldsymbol{\omega}_m)}(\overline{t}),K_{\sigma^k(\boldsymbol{\eta}_1 \boldsymbol{\eta}_2\dots \boldsymbol{\eta}_m) }(\overline{s})\big)\Big) ~~f(\overline{t},\overline{s})\\& - \Big(\prod_{k=1}^{m} \alpha_{\boldsymbol{\omega}_k, \eta_k}\big(L_{\sigma ^{k}(\boldsymbol{\omega}_1 \boldsymbol{\omega}_2\dots \boldsymbol{\omega}_m)}(\overline{t'}),K_{\sigma^k(\boldsymbol{\eta}_1 \boldsymbol{\eta}_2\dots \boldsymbol{\eta}_m) }(\overline{s'})\big)\Big) ~~f(\overline{t'},\overline{s'}) \Bigg\|_2\\
                 &+    \sum_{r=1}^{m}\Bigg\| \Big(\prod_{k=1}^{r-1} \alpha_{\boldsymbol{\omega}_k \boldsymbol{\eta}_k}(L_{\sigma^k(\boldsymbol{\omega}_1 \boldsymbol{\omega}_2\dots \boldsymbol{\omega}_m) }(\overline{t}),K_{\sigma^k(\boldsymbol{\eta}_1 \boldsymbol{\eta}_2\dots \boldsymbol{\eta}_m) }(\overline{s}))\Big)h_{\boldsymbol{\omega}_r\boldsymbol{\eta}_r}\big(L_{\sigma ^{r-1}(\boldsymbol{\omega}_1 \boldsymbol{\omega}_2\dots \boldsymbol{\omega}_m)}(\overline{t}),K_{\sigma ^{r-1}(\boldsymbol{\eta}_1 \boldsymbol{\eta}_2\dots \boldsymbol{\eta}_m)}(\overline{s})\big)\\&- \Big(\prod_{k=1}^{r-1} \alpha_{\boldsymbol{\omega}_k \boldsymbol{\eta}_k}(L_{\sigma^k(\boldsymbol{\omega}_1 \boldsymbol{\omega}_2\dots \boldsymbol{\omega}_m) }(\overline{t'}),K_{\sigma^k(\boldsymbol{\eta}_1 \boldsymbol{\eta}_2\dots \boldsymbol{\eta}_m) }(\overline{s'}))\Big)h_{\boldsymbol{\omega}_r\boldsymbol{\eta}_r}\big(L_{\sigma ^{r-1}(\boldsymbol{\omega}_1 \boldsymbol{\omega}_2\dots \boldsymbol{\omega}_m)}(\overline{t'}),K_{\sigma ^{r-1}(\boldsymbol{\eta}_1 \boldsymbol{\eta}_2\dots \boldsymbol{\eta}_m)}(\overline{s'})\big)  \Bigg\|                     \\
                 \le &
                 \sum_{r=1}^{m} \Big|\Big(\prod_{k=1}^{r-1} \alpha_{\boldsymbol{\omega}_k \boldsymbol{\eta}_k}(L_{\sigma^k(\boldsymbol{\omega}_1 \boldsymbol{\omega}_2\dots \boldsymbol{\omega}_m) }(\overline{t}),K_{\sigma^k(\boldsymbol{\eta}_1 \boldsymbol{\eta}_2\dots \boldsymbol{\eta}_m) }(\overline{s}))\Big)\Big|~~\Big\| h_{\boldsymbol{\omega}_r\boldsymbol{\eta}_r}\big(L_{\sigma ^{r-1}(\boldsymbol{\omega}_1 \boldsymbol{\omega}_2\dots \boldsymbol{\omega}_m)}(\overline{t}),K_{\sigma ^{r-1}(\boldsymbol{\eta}_1 \boldsymbol{\eta}_2\dots \boldsymbol{\eta}_m)}(\overline{s})\big)\\&- h_{\boldsymbol{\omega}_r\boldsymbol{\eta}_r}\big(L_{\sigma ^{r-1}(\boldsymbol{\omega}_1 \boldsymbol{\omega}_2\dots \boldsymbol{\omega}_m)}(\overline{t'}),K_{\sigma ^{r-1}(\boldsymbol{\eta}_1 \boldsymbol{\eta}_2\dots \boldsymbol{\eta}_m)}(\overline{s'})\big)\Big\|_2 \\&+
                 \sum_{r=1}^{m}\Big|\Big(\prod_{k=1}^{r-1} \alpha_{\boldsymbol{\omega}_k \boldsymbol{\eta}_k}(L_{\sigma^k(\boldsymbol{\omega}_1 \boldsymbol{\omega}_2\dots \boldsymbol{\omega}_m) }(\overline{t}),K_{\sigma^k(\boldsymbol{\eta}_1 \boldsymbol{\eta}_2\dots \boldsymbol{\eta}_m) }(\overline{s}))\Big)\\& - \Big(\prod_{k=1}^{r-1} \alpha_{\boldsymbol{\omega}_k \boldsymbol{\eta}_k}(L_{\sigma^k(\boldsymbol{\omega}_1 \boldsymbol{\omega}_2\dots \boldsymbol{\omega}_m) }(\overline{t'}),K_{\sigma^k(\boldsymbol{\eta}_1 \boldsymbol{\eta}_2\dots \boldsymbol{\eta}_m) }(\overline{s'}))\Big) \Big|\\& \big\| h_{\boldsymbol{\omega}_r\boldsymbol{\eta}_r}\big(L_{\sigma ^{r-1}(\boldsymbol{\omega}_1 \boldsymbol{\omega}_2\dots \boldsymbol{\omega}_m)}(\overline{t'}),K_{\sigma ^{r-1}(\boldsymbol{\eta}_1 \boldsymbol{\eta}_2\dots \boldsymbol{\eta}_m)}(\overline{s'})\big) \big\|_2 \\&+ 2(\| \alpha\|_{\infty})^m \|f\|_{\infty}.
                  \end{aligned}
               \end{equation*}
 Set $s=\min\big\{s_h,s_ {\alpha}\big\}.$ Using $\|L_{\sigma ^r(\boldsymbol{\omega}_1 \boldsymbol{\omega}_2\dots \boldsymbol{\omega}_m)}(\overline{t})- L_{\sigma ^r(\boldsymbol{\omega}_1 \boldsymbol{\omega}_2\dots \boldsymbol{\omega}_m)}(\overline{t'})\|_2 = a^{-r}  \|t-t'\|_2$ and H\"{o}lder continuity of $h_{\boldsymbol{\omega}_r\boldsymbol{\eta}_r}$, we have
\begin{equation*}
               \begin{aligned}
                & \Big\| h_{\boldsymbol{\omega}_r\boldsymbol{\eta}_r}\big(L_{\sigma ^{r-1}(\boldsymbol{\omega}_1 \boldsymbol{\omega}_2\dots \boldsymbol{\omega}_m)}(\overline{t}),K_{\sigma ^{r-1}(\boldsymbol{\eta}_1 \boldsymbol{\eta}_2\dots \boldsymbol{\eta}_m)}(\overline{s})\big)- h_{\boldsymbol{\omega}_r\boldsymbol{\eta}_r}\big(L_{\sigma ^{r-1}(\boldsymbol{\omega}_1 \boldsymbol{\omega}_2\dots \boldsymbol{\omega}_m)}(\overline{t'}),K_{\sigma ^{r-1}(\boldsymbol{\eta}_1 \boldsymbol{\eta}_2\dots \boldsymbol{\eta}_m)}(\overline{s'})\big)\Big\|_2 \\&\le ~ K_h \Big\|\big(L_{\sigma ^{r-1}(\boldsymbol{\omega}_1 \boldsymbol{\omega}_2\dots \boldsymbol{\omega}_m)}(\overline{t}),K_{\sigma ^{r-1}(\boldsymbol{\eta}_1 \boldsymbol{\eta}_2\dots \boldsymbol{\eta}_m)}(\overline{s})\big)-\big(L_{\sigma ^{r-1}(\boldsymbol{\omega}_1 \boldsymbol{\omega}_2\dots \boldsymbol{\omega}_m)}(\overline{t'}),K_{\sigma ^{r-1}(\boldsymbol{\eta}_1 \boldsymbol{\eta}_2\dots \boldsymbol{\eta}_m)}(\overline{s'})\big)\Big\|_2^{s_h}\\
                 =~& K_h \Big[a^{-(r-1)} \|(t,s)- (t',s')\|_2 \Big]^{s_h}\\
                 \le~& K_h a^{-(r-1)} \|(t,s)- (t',s')\|_2^s.
               \end{aligned}
               \end{equation*}
  Now, let $K_{\alpha}=\max \{K_{\alpha_{\boldsymbol{\omega}\boldsymbol{\eta}}}:  \boldsymbol{\omega},\boldsymbol{\eta} \in I^N \}$. Since $\alpha_{\boldsymbol{\omega}\boldsymbol{\eta}}$ is H\"{o}lder continuous with exponent $s_{\alpha_{\boldsymbol{\omega}\boldsymbol{\eta}}}$ and H\"{o}lder constant $K_{\alpha_{\boldsymbol{\omega}\boldsymbol{\eta}}}$, in view of Lemma \ref{Inequality}, for $r\ge 2,$ we have 
\begin{equation*}
               \begin{aligned}
                & \Big|\prod_{k=1}^{r-1} \alpha_{\boldsymbol{\omega}_k,\boldsymbol{\eta}_k}\big(L_{\sigma ^{k}(\boldsymbol{\omega}_1 \boldsymbol{\omega}_2\dots \boldsymbol{\omega}_m)}(\overline{t}),K_{\sigma ^{k}(\boldsymbol{\eta}_1 \boldsymbol{\eta}_2\dots \boldsymbol{\eta}_m)}(\overline{s})\big) - \prod_{k=1}^{r-1} \alpha_{\boldsymbol{\omega}_k,\boldsymbol{\eta}_k}\big(L_{\sigma ^{k}(\boldsymbol{\omega}_1 \boldsymbol{\omega}_2\dots \boldsymbol{\omega}_m)}(\overline{t'}),K_{\sigma ^{k}(\boldsymbol{\eta}_1 \boldsymbol{\eta}_2\dots \boldsymbol{\eta}_m)}(\overline{s'})\big) \Big|\\ \le& \sum_{k=1}^{r-1} (\| \alpha\|_{\infty})^{r-2} K_\alpha \Big\| \big(L_{\sigma ^{k}(\boldsymbol{\omega}_1 \boldsymbol{\omega}_2\dots \boldsymbol{\omega}_m)}(\overline{t}),K_{\sigma ^{k}(\boldsymbol{\eta}_1 \boldsymbol{\eta}_2\dots \boldsymbol{\eta}_m)}(\overline{s})\big)- \big(L_{\sigma ^{k}(\boldsymbol{\omega}_1 \boldsymbol{\omega}_2\dots \boldsymbol{\omega}_m)}(\overline{t'}),K_{\sigma ^{k}(\boldsymbol{\eta}_1 \boldsymbol{\eta}_2\dots \boldsymbol{\eta}_m)}(\overline{s'})\big) \Big\|_2^{s_\alpha}\\
                 =~& \sum_{k=1}^{r-1} (\| \alpha\|_{\infty})^{r-2} K_\alpha \Big[a^{-k} \|(t,s)- (t',s')\|_2 \Big]^{s_\alpha}\\
                 \le ~& K_\alpha (\| \alpha\|_{\infty})^{r-2} \|(t,s)- (t',s')\|_2^s \sum_{k=1}^{r-1}a^{-k}\\
                 \le ~& K_\alpha (\| \alpha\|_{\infty})^{r-2} \|(t,s)- (t',s')\|_2^s \frac{a^{1-r}}{1-a}\\
                 =~& \frac{K_\alpha \delta^{r-2}}{a(1-a)} \|(t,s)-(t',s')\|_2^s.
               \end{aligned}
               \end{equation*}
Using the above estimates, we get
\begin{equation*}
               \begin{aligned}
                  \big|f(t,s)- f(t',s')\big|  \le &   \sum_{r=1}^{m} \big(\| \alpha\|_{\infty}\big)^{r-1} K_h a^{-(r-1)} \|(t,s)-(t',s')\|_2^s\\ &+\sum_{r=1}^{m}\frac{K_\alpha \|h\|_{\infty}\delta^{r-2}}{a(1-a)} \|(t,s)-(t',s')\|_2^s +2\big(\| \alpha\|_{\infty}\big)^m \|f\|_{\infty}\\
                  =&  \sum_{r=1}^{m} \big(\| \alpha\|_{\infty}\big)^{r-1} K_h a^{-(r-1)} \|(t,s)-(t',s')\|_2^s\\ &+ \sum_{r=2}^{m}\frac{K_\alpha \|h\|_{\infty}\delta^{r-2}}{a(1-a)} \|(t,s)-(t',s')\|_2^s +2 \|f\|_{\infty} \delta^m a^m.\\
                  \end{aligned}
               \end{equation*}
Since $a^{m+1} \le \|(t,s)- (t',s')\|_2,$ we have
\begin{equation*}
               \begin{aligned}
                  ~\big \|f(t,s)- f(t',s')\big\|_2 \le &  \sum_{r=1}^{m} \delta^{(r-1) }K_h  \|(t,s)- (t',s')\|_2^s\\ &+\sum_{r=2}^{m}\frac{K_\alpha \|h\|_{\infty}\delta^{r-2}}{a(1-a)} \|(t,s)- (t',s')\|_2^s + \frac{2 \|f\|_{\infty} \delta^m}{a} \|(t,s)- (t',s')\|_2 \\
                  \le&~ 4K \|(t,s)-(t',s')\|_2^s \sum_{r=1}^{m+1} \delta^{r-1},
               \end{aligned}
               \end{equation*}
               for a fixed suitable constant $K$.
 \begin{itemize}
 
 \item[Case 1:]
 If $ \delta < 1$, then we have $$\sum_{r=1}^{m+1}\delta^{r-1} \le \frac{1}{1-\delta}$$ and hence we obtain $$ |f(t,s)- f(t',s')|\le \frac{ 4K}{1-\delta} \|(t,s) -(t',s')\|_2^s.$$ That is, $f$ is  H\"older continuous with exponent $s.$
 \item[Case 2:] If $\delta = 1$, then $$ \|f(t,s)- f(t',s')\|_2 \le  4K(m+1) \|(t,s) - (t',s')\|_2^s.$$ Since $\|(t,s) -(t',s')\|_2 \le a^m< 1,$  we get $m \le \ln(\|(t,s)-(t',s')\|_2)/\ln(a) .$ With the help of a known inequality: $$0 < -x^{\mu} \ln x \le \dfrac{1}{\mu e}, ~ \text{for}~0< x \le 1~\text{ and}~ 0 < \mu < 1,$$ we select a suitable $0 < \mu < 1$ such that 
 \begin{equation*}
                \begin{aligned} 
                (m+1)\|(t,s)-(t',s')\|_2^s \le ~ & \Big(1+\frac{\ln(\|(t,s)-(t',s')\|_2)}{\ln a}\Big)\|(t,s)-(t',s')\|_2^s\\
                   = ~& \|(t,s)- (t',s')\|_2^s\\& + \frac{-\|(t,s)-(t',s')\|_2^{\mu}~\ln(\|(t,s)-(t',s')\|_2)}{|\ln a|}\|(t,s)-(t',s')\|_2^{s-\mu}\\
                   \le ~& \|(t,s)- (t',s')\|_2^s + \frac{1}{ \mu e |\ln a|}\|(t,s)-(t',s')\|_2^{s-\mu}\\
                   \le~ & \Big(1+\frac{1}{ \mu e~ |\ln a|}\Big) \|(t,s)-(t',s')\|_2^{s-\mu}.
                   \end{aligned} 
                \end{equation*}
 Hence, we have $$ |f(t,s)- f(t',s')| \le  4K \Big(1+\dfrac{1}{ \mu e |\ln a|}\Big) \|(t,s)-(t',s')\|_2^{s-\mu}.$$
 \item[Case 3:] If $\delta > 1$, then $$ |f(t,s)- f(t',s')| \le  4K \frac{\delta^{m+1}}{\delta-1}~~ \|(t,s) -(t',s')\|_2^s.$$ We now consider a positive number $\lambda$ with $0< \lambda <1$ such that $$ \delta^{m+1} \|(t,s) - (t',s')\|_2^s \le \|(t,s) - (t',s')\|_2^{\lambda}. $$ Further, one obtains $$ \lambda \le s+ \frac{(m+1)\ln \delta}{ \ln(\|(t,s) - (t',s')\|_2)}.$$ Since $a^{m+1} \le \|(t,s) - (t',s')\|_2,$ we obtain $ \dfrac{1}{\ln(\|(t,s) - (t',s')\|_2)} \le \dfrac{1}{(m+1) \ln(a)}.$ This in turn yields $\lambda \le s+ \dfrac{\ln \delta}{ \ln a}=s-1+\frac{\ln \|\alpha\|_{\infty}}{ \ln a} < 1.$ Therefore, we get $$|f(t,s)- f(t',s')| \le   ~~ \frac{4K}{\delta-1}~~ \|(t,s) -(t',s')\|_2^{\lambda}.$$
 \end{itemize}
Thus, the proof of the theorem is completed.
\end{proof}
\section{Dimension of FIF}\label{sec:fd}
In this section, we focus on estimating fractal dimension such as Hausdorff dimension and box dimension of the constructed FIFs. Here, as usual we denote by $ \dim_H(A)$ and $\dim_B(A)$ the Hausdorff dimension and box dimension of a set $A$ respectively. For more details and definitions of these dimensions, we encourage the reader to see the book \cite{Fal} of Falconer.
\begin{theorem}
Under the hypotheses of Theorem \ref{Holder}, if $\|\alpha\|_{\infty} < \frac{1}{2^N}$ then $$2 ~\frac{\log3}{\log2} \le \dim_H(Gr(f)) \le \overline{\dim}_B(Gr(f)) \le 1-s +2 ~\frac{\log3}{\log2}.$$
\end{theorem}
\begin{proof}
Let us first define the maximum range of the function $f$ over $X \subseteq SG' \times SG''$ as $R_f[X]= \sup_{(t,s),(t',s')\in X} |f(t,s)-f(t',s')|.$
In view of Theorem \ref{Holder}, $f$ satisfies H\"{o}lder condition, we have $$ R_f[L_{\boldsymbol{\omega}}(SG') \times K_{\boldsymbol{\eta}}(SG'')] \le  \frac{c}{2^{ns}}~~\text{for every}~\boldsymbol{\omega},\boldsymbol{\eta} \in I^n,$$ where $c$ is suitable constant depending the H\"older constant of $f$ and $s=\min\{s_h,s_{\alpha}\}.$ Suppose that $ \delta = \frac{1}{2^n}$ for some $n \in \mathbb{N}.$ If $N_{\delta}(Gr(f))$ denotes the number of $\delta-$cubes that intersect graph of $f,$ then $$  N_{\delta}(Gr(f)) \le 2.3^n.3^n + 2^n \sum_{\boldsymbol{\omega},\boldsymbol{\eta} \in I^n} R_f[L_{\boldsymbol{\omega}}(SG') \times K_{\boldsymbol{\eta}}(SG'')].$$We obtain $N_{\delta}(Gr(f)) \le 2.3^n.3^n + c2^{n(1-s)}3^n3^n .$
Upper box-dimension of $Gr(f)$ can be estimated in the following way $$ \overline{\lim}_{\delta \rightarrow 0} \frac{\log N_{\delta}(Gr(f))}{-\log \delta} \le \lim_{n \rightarrow \infty}\frac{\log(2.3^n3^n + c2^{n(1-s)}3^n3^n)}{\log 2^n}$$
which produces $ \overline{\lim}_{\delta \rightarrow 0} \frac{\log N_{\delta}(Gr(f))}{-\log \delta} \le 1-s +2~\frac{\log3}{\log2}.$ 

For the lower bound, we begin by defining a mapping $T : Gr(f) \rightarrow SG' \times SG'' $ by $T\big((t,s,f(t,s))\big)=(t,s).$ Then
\begin{equation*}
    \begin{aligned}
    \|T\big((t,s,f(t,s))\big)-T\big((t',s',f(t',s'))\big)\|_2&=\|(t,s)-(t',s')\|_2\\& \le \sqrt{\|(t,s)-(t',s')\|_2^2+ (f(t,s)-f(t',s'))^2}\\&= \|(t,s,f(t,s))-(t',s',f(t',s'))\|_2,
    \end{aligned}
\end{equation*}
implies that $T$ is Lipschitz. Now, for any $ (t,s) \in SG' \times SG''$, we can choose $(t,s,f(t,s)) \in Gr(f)$ such that $(t,s)=T\big((t,s,f(t,s))\big)$ which implies the surjectiveness of $T$. Recall (Cf. \cite[Corollary 2.4(a)]{Fal}) that if $g:A\subseteq \mathbb{R}^n \to \mathbb{R}^m$ is a Lipschitz function,  then $\dim_{H}\big( g(A) \big) \le \dim_{H}(A).$ Using \cite[Corollary 7.4]{Fal}, it is immediate that $\dim_H \big(SG' \times SG'')=2~\frac{\log3}{\log2}$. In view of the previous two observations, we get $$ \dim_H\big(Gr(f)\big) \ge \dim_H \big(T(Gr(f))\big)= \dim_H \big(SG' \times SG'')=2~\frac{\log3}{\log2},$$
proving the assertion.
\end{proof}

\bibliographystyle{amsplain}

\begin{thebibliography}{10}
\bibitem{EPJST} V. Agrawal, T. Som, Fractal dimension of $\alpha$-fractal function on the Sierpi\'nski Gasket, Eur. Phys. J. Spec. Top. (2021) https://doi.org/10.1140/epjs/s11734-021-00304-9.

\bibitem{RIM} V. Agrawal, T. Som, $L^{p}$ approximation using fractal functions on the Sierpi\'nski gasket, Results Math 77, 74 (2022) https://doi.org/10.1007/s00025-021-01565-5.
\bibitem{MF2} M. F. Barnsley, Fractals Everywhere, Academic Press, Orlando, Florida, 1988. 

\bibitem{baldo94} Baldo S., Normant F., Tricot C., Fractals in Engineering, Montr\'eal Proceedings, World Scientific, 1994.
\bibitem{SS2} S. Chandra, S. Abbas, The calculus of bivariate fractal interpolation surfaces, Fractals 29(3) (2020) 2150066.
\bibitem{havlina95} Havlina S. et. al., Fractals in biology and medicine. Chaos Solitons Fractals, 1995, 6, 171--201.

\bibitem{iovane04} Iovane G., Laserra E., Tortoriello F.S., Stochastic self-similar and fractal universe, Chaos Solitons Fractals, 2004, 20, 415-426.

\bibitem{voss89} Voss R.F., Random fractals: Self affinity in noise, music, mountains, and clouds, Phys. D, 1989, 38, 362-371.

\bibitem{barnsley89_2} Barnsley M.F., Elton J., Hardin D., Massopust P., Hidden variable fractal interpolation functions, SIAM J. Math. Anal., 1989, 20, 5, 1218-1242.

\bibitem{chand07} Chand A.K.B., Kapoor G.P., Smoothness analysis of coalescence hidden variable fractal interpolation functions, Int. J. Nonlinear Sci., 2007, 3, 15-26.
\bibitem{Fal} Falconer K. J., Fractal Geometry: Mathematical Foundations and Applications, John Wiley Sons Inc., New York, 1990.
\bibitem{Jha} S. Jha and S. Verma, Dimensional Analysis of  $\alpha$ -Fractal Functions, Results in Mathematics 76 (4)  (2021) 1-24.
\bibitem{navascues04} Navascu\'es M.A., Fractal polynomial interpolation. J. Anal. App., 2005, 24, no.2,  401-418.

\bibitem{navas-sebas04} Navascu\'es M.A., Sebasti\'an M.V., Generalization of hermite functions by fractal interpolation, J. Approx. Theory, 2004, 131, 19-29.

\bibitem{srijanani15_1} Kapoor G.P., Prasad S.A., Super fractal interpolation functions, Int. J. Nonlinear Sci., 2015, 19, 1, 20-29.

\bibitem{srijanani21} Prasad S.A., Super coalescence hidden variable fractal interpolation functions, Fractals, 2021, 29, 3, Article ID 2150051.

\bibitem{gang96} Gang C., The smoothness and dimension of fractal interpolation function, Appl. Math. J. Chinese Univ. Ser. B, 1996, 11, 409.

\bibitem{navascues03} Navascu\'es M.A., Sebastian M.V., Some results of convergence of cubic spline fractal interpolation functions, Fractals, 2003, 11, 1, 1-7.
\bibitem {NVV2}  M. A. Navascu\'es, S. Verma, and P. Viswanathan, Concerning the vector-valued fractal interpolation functions on the Sierpi\'nski gasket, Mediterr. J. Math .18, article no. 202(2021).
\bibitem{srijanani13_2} Prasad S.A., Regularity of fractal interpolation function via wavelet transforms, Adv. Pure Appl. Math., 2013, 4, 2, 189-202.

\bibitem{srijanani13_1} Prasad S.A., Node insertion in coalescence fractal interpolation function, Chaos Solitons Fractals, 2013, 49, 16-20.

\bibitem{srijanani14_1} Kapoor G.P., Prasad S.A., Convergence of cubic spline super fractal interpolation functions, Fractals, 2014, 22, 1, 1-7.

\bibitem{srijanani14_2} Kapoor G.P., Prasad S.A., Multiresolution analysis based on coalescence hidden-variable fractal interpolation functions, Int. J. Comput. Math., 2014, Article ID 531562.

\bibitem{pan14} XueZai Pan, Fractional Calculus of Fractal Interpolation Function on $[0,b]$.
Abstr. Appl. Anal., 2014, Article ID: 640628.

\bibitem{ruan09} Huo-Jun Ruan and Wei-Yi Su and Kui Yao, Box dimension and fractional integral of linear fractal interpolation functions. J. Approx. Theory, 2009, 161, 187-197.

\bibitem{srijanani17} Prasad S.A., Fractional calculus of coalescence hidden-variable fractal interpolation functions, Fractals, 2017, 25, 2, Article ID 1750019.

\bibitem{massopust90} Massopust P., Fractal surfaces. J. Math. Anal. Appl., 1990, 151, 275-290.
\bibitem{AKBCK} A.K.B. Chand, G.P. Kapoor, Generalized cubic spline fractal interpolation functions, SIAM J. Numer. Anal. 44(2006) 655-676.
\bibitem{geronimo93}
J.S. Geronimo, D.P. Hardin, Fractal interpolation surfaces and a related 2-{D} multiresolution analysis. J. Math. Anal. Appl.,  1993, 176, 561-586.

\bibitem{massopust93} Massopust P., Hardin D., Fractal interpolation functions from $\mathbb{R}^n$ into $\mathbb{R}^m$ and their projections.
J. Anal. App., 1993, 12, No. 3, 535-548.

\bibitem{dalla02} Dalla L., Bivariate fractal interpolation functions on grids. Fractals, 2002, 10, No. 1, 53-58.


\bibitem {Celik} D. Celik, S. Kocak, Y. \"Ozdemir, Fractal interpolation on the Sierpi\'nski Gasket, J. Math. Anal. Appl., 2008,  337, 343-347.

\bibitem {Ruan3} H.-J. Ruan, Fractal interpolation functions on post critically finite self-similar sets, Fractals, 2010,  18, 119-125.
\bibitem{Ruan} Huo-Jun Ruan and Qiang Xu, Fractal interpolation surfaces on Rectangular Grids, Bull. Aust. Math. Soc., 91 (2015)  435-446.
\bibitem {Ruan4} S.-G. Ri, H.-J. Ruan, Some properties of fractal interpolation functions on Sierpinski gasket, J. Math. Anal. Appl., 2011, 380, 313-322.

\bibitem{SP} A. Sahu, A. Priyadarshi, On the box-counting dimension of graphs of harmonic functions on the Sierpi\'{n}ski gasket, J. Math. Anal. Appl., 2020, 487, Article ID: 124036. 

\bibitem{WY} H.-Y. Wang, J.-S. Yu, Fractal interpolation functions with variable parameters and their analytical properties, J. Approx. Theory, 2013,  175, 1-18.

\bibitem{BS}  B. Bockelman and R. S. Strichartz, Partial Differential Equations on Products of Sierpinski Gaskets, Indiana University Mathematics Journal, 2007, 56, No. 3, 1361-1375.
\bibitem{Stri} R. S. Strichartz, Analysis on products of fractals. Trans. Amer. Math. Soc. 357 (2005), no. 2, 571–615.
\bibitem{Ver1} M. Verma, A. Priyadarshi, S. Verma, Vector-valued fractal functions: Fractal dimension and Fractional calculus, https://doi.org/10.48550/arXiv.2205.00892 (2022).

\end{thebibliography}

\end{document}